\documentclass{amsart}
\usepackage{amssymb}
\usepackage{amsmath}
\usepackage{amsfonts}

\theoremstyle{plain}

\newtheorem{corollary}{Corollary}

\newtheorem{definition}{Definition}

\newtheorem{proposition}{Proposition}

\numberwithin{equation}{section}
\input{tcilatex}

\begin{document}
\title[Norm Estimates]{Norm Estimates for Solutions of Elliptic BVPs of the
Dirac Operator}
\author{Dejenie A. Lakew}
\address{John Tyler Community College\\
Department of Mathematics}
\email{dlakew@jtcc.edu}
\urladdr{http://www.jtcc.edu}
\date{December 30, 2013}
\thanks{This paper is in final form and no version of it will be submitted
for publication elsewhere.}

\begin{abstract}
We present norm estimates for solutions of first and second order elliptic
BVPs of the Dirac operator $D=\sum_{j=1}^{n}e_{j}\partial _{x_{j}}$
considered over bounded and smooth domain $\Omega $ of $%
\mathbb{R}
^{n}$. The solutions whose norms to be estimated are in some Sobolev spaces $%
W^{k,p}\left( \Omega \right) $ and the boundary conditions as traces of
solutions and their derivatives are in some Slobodeckij spaces $W^{\lambda
,p}\left( \partial \Omega \right) $ where $\lambda $ is some non integer but
fractional number, for $1\leq p<\infty $ and $k\in 
\mathbb{Z}
$.
\end{abstract}

\maketitle

\section{\protect\bigskip \textbf{Algebraic and Analytic Rudiments of }$%
Cl_{n}$}

Let $\{e_{j}:j=1,2,...,n\}$ be an orthonormal basis for $%
\mathbb{R}
^{n}$ that is equipped with an inner product so that 
\begin{equation}
e_{i}e_{j}+e_{j}e_{i}=-2\delta _{ij}e_{0}  \label{inner product 1}
\end{equation}%
where $\delta _{ij}$ is the Kronecker delta. The inner product satisfies an
anti commutative relation 
\begin{equation}
x^{2}=-\Vert x\Vert ^{2}  \label{inner product 2}
\end{equation}

Therefore $%
\mathbb{R}
^{n}$ with these properties of base vectors generates a non commutative
algebra called Clifford algebra denoted by $Cl_{n}$.

\ 

The basis of $Cl_{n}$ will then be 
\begin{equation*}
\{e_{A}:A\subset \{1<2<3<...<n\}\}
\end{equation*}%
which implies: 
\begin{equation*}
\dim (Cl_{n})=2^{n}
\end{equation*}

The object $e_{0}$ used above is the identity element of the Clifford
algebra $Cl_{n}$. \ 

\ 

Representation of elemnets of $Cl_{n}$: every $a\in Cl_{n}$ is represented by%
\begin{equation}
a=\sum e_{A}a_{A}  \label{Clifford element}
\end{equation}%
where $a_{A}$ is a real number.

\ 

Thus every $x=(x_{1},x_{2},...,x_{n})\in 
\mathbb{R}
^{n}$ can be identified with $\sum_{j=1}^{n}e_{j}x_{j}$ of $Cl_{n}$ and
therefore we have an embedding 
\begin{equation*}
\mathbb{R}
^{n}\hookrightarrow Cl_{n}
\end{equation*}

We also define what is called a Clifford conjugate of 
\begin{equation*}
a=\sum e_{A}a_{A}
\end{equation*}%
as 
\begin{equation*}
\overline{a}=\sum \overline{e}_{A}a_{A}
\end{equation*}%
where%
\begin{equation*}
\overline{e_{j_{1}}...e}_{j_{r}}=\left( -1\right) ^{r}e_{j_{r}}...e_{j_{1}}
\end{equation*}

For instance for $i,j=1,2,...,n$,

\begin{equation*}
\overline{e}_{j}=-e_{j},\text{ \ }e_{j}^{2}=-1
\end{equation*}%
and for 
\begin{equation*}
i\neq j:\overline{e_{i}e}_{j}=(-1)^{2}e_{j}e_{i}=e_{j}e_{i}
\end{equation*}

\begin{definition}
\bigskip We define the Clifford norm of 
\begin{equation*}
a=\sum e_{A}a_{A}\in Cl_{n}
\end{equation*}%
by 
\begin{equation}
\Vert a\Vert =\left( \left( a\overline{a}\right) _{0}\right) ^{\frac{1}{2}%
}=\left( \underset{A}{\sum a_{A}^{2}}\right) ^{\frac{1}{2}}
\label{Clifford Norm}
\end{equation}%
where $\left( a\right) _{0}$ is the real part of $a\overline{a}$.
\end{definition}

\ \ \ \ \ \ \ \ \ \ \ \ \ \ \ 

The norm $\Vert .\Vert $ satisfies the inequality: 
\begin{equation}
\Vert ab\Vert \leq c\left( n\right) \Vert a\Vert \Vert b\Vert
\label{Norm Inequality}
\end{equation}%
with $c\left( n\right) $ a dimensional constant.

\ 

Also each non zero element $x\in 
\mathbb{R}
^{n}$ has an inverse given by :%
\begin{equation}
x^{-1}=\frac{\overline{x}}{\Vert x\Vert ^{2}}  \label{inverse}
\end{equation}

$\sphericalangle $ \ In the article it is always the case that $1<p<\infty $
unless otherwise specified and $\Omega $ is a bounded and smooth (at least
with $C^{1}$ - boundary $\partial \Omega $) domain of $%
\mathbb{R}
^{n}$

\ \ \ \ 

\bigskip A Clifford valued ($Cl_{n}$- valued) function $f$ defined on $%
\Omega $ as%
\begin{equation*}
f:\Omega \longrightarrow Cl_{n}
\end{equation*}%
has a representation $\ $

\begin{equation}
\ f=\sum_{A}e_{A}f_{A}  \label{Clifford valued function}
\end{equation}%
where $f_{A}:\Omega \longrightarrow 
\mathbb{R}
$ is a real valued component or section of $f$.

\ \ 

\begin{definition}
For a function $f\in C^{1}\left( \Omega \right) \cap C\left( \overline{%
\Omega }\right) $, we define the Dirac derivative of $f$ by 
\begin{equation}
Df\left( x\right) =\sum_{j=1}^{n}e_{j}\partial _{x_{j}}f\left( x\right)
\label{Dirac}
\end{equation}

A function $f:\Omega \longrightarrow Cl_{n}$ is called left monogenic or
left Clifford analytic over $\Omega $ if 
\begin{equation*}
Df\left( x\right) =0,\text{ }\forall x\in \Omega
\end{equation*}%
and likewise it is called right monogenic over $\Omega $ if 
\begin{equation*}
f(x)D=\sum_{j=1}^{n}\partial _{x_{j}}f\left( x\right) e_{j}=0,\text{ }%
\forall x\in \Omega
\end{equation*}
\end{definition}

An example of both left and right monogenic function defined over $%
\mathbb{R}
^{n}\backslash \{0\}$ is given by 
\begin{equation*}
\psi \left( x\right) =\frac{\overline{x}}{\omega _{n}\Vert x\Vert ^{n}}
\end{equation*}%
where $\omega _{n}$ is the surface area of the unit sphere in $%
\mathbb{R}
^{n}$.

\ \ 

The function $\psi $\ is also a fundamental solution to the Dirac operator $%
D $ and we define integral transforms as convolutions of $\psi $ with
functions of some function spaces below.

\ 

\begin{definition}
Let $f\in C^{1}\left( \Omega ,Cl_{n}\right) \cap C\left( \overline{\Omega }%
\right) $.

We define two integral transforms as follow:

\begin{equation}
\zeta _{\Omega }f\left( x\right) =\int_{\Omega }\psi \left( y-x\right)
f\left( y\right) d\Omega _{y},\text{ \ }x\in \Omega  \label{Theodorescu}
\end{equation}

\begin{equation}
\xi _{\partial \Omega }f\left( x\right) =\int_{\partial \Omega }\psi \left(
y-x\right) \upsilon \left( y\right) f\left( y\right) d\partial \Omega _{y},%
\text{ \ }x\notin \partial \Omega  \label{Feuter}
\end{equation}
\end{definition}

\ \ \ \ \ \ \ \ \ \ \ \ \ \ \ \ \ \ \ 

The integral transform defined in $\left( \ref{Theodorescu}\right) $ a
domain integral is called the Theodorescu transform or the Cauchy transform.
It is a convolution $\psi \ast f$ over $\Omega $. The integral transform
defined in $\left( \ref{Feuter}\right) $ is some times called the \textit{%
Feuter} transform as a boundary integral which again is a convolution $\psi
\ast \upsilon f$ over $\partial \Omega $. $\upsilon \left( y\right) $ is a
unit normal vector pointing outward at $y\in \partial \Omega $.

\ 

\section{Sobolev and Slobodeckij Spaces}

\ \ 

\begin{definition}
For $1<p<\infty $, $k\in 
\mathbb{N}
\cup \{0\}$ we define:
\end{definition}

\begin{description}
\item[I] The Sobolev space $W^{k,p}\left( \Omega \right) $ as%
\begin{equation*}
W^{k,p}\left( \Omega \right) :=\{f\in \tciLaplace ^{p}\left( \Omega \right)
:D^{\alpha }f\in \tciLaplace ^{p}\left( \Omega \right) ,\text{ \ }\Vert
\alpha \Vert \leq k\}
\end{equation*}%
with norm $\ $%
\begin{equation}
\Vert f\Vert _{W^{k,p}\left( \Omega \right) }=\left( \underset{\Vert \alpha
\Vert \leq k}{\sum }\int_{\Omega }|D^{\alpha }f|^{p}dx\right) ^{\frac{1}{p}}
\label{Sobolev Norm}
\end{equation}

\item[II] The Slobodeckij spaces for $0<\lambda <1$ as%
\begin{equation*}
W^{\lambda ,p}\left( \partial \Omega \right) :=\{f\in \tciLaplace ^{p}\left(
\partial \Omega \right) :\int_{\partial \Omega }\int_{\partial \Omega }\frac{%
|f\left( x\right) -f\left( y\right) |^{p}}{|x-y|^{n+\lambda p-1}}d\sigma
_{x}d_{\sigma y}<\infty \}
\end{equation*}%
and norm is defined by $\ $%
\begin{equation}
\Vert f\Vert _{W^{\lambda ,p}\left( \partial \Omega \right) }=\left(
\int_{\partial \Omega }\int_{\partial \Omega }\frac{|f\left( x\right)
-f\left( y\right) |^{p}}{|x-y|^{n+\lambda p-1}}d\sigma _{x}d_{\sigma
y}\right) ^{\frac{1}{p}}  \label{Slobodeckij Norm}
\end{equation}

\item[III] The Slobodeckij spaces for $\lambda =[\lambda ]+\{\lambda \}$
where $0<\{\lambda \}<1$ :%
\begin{equation*}
W^{\lambda ,p}\left( \partial \Omega \right) :=\{f\in W^{[\lambda ],p}\left(
\partial \Omega \right) :\underset{\Vert \alpha \Vert \leq \lbrack \lambda ]}%
{\sum }\int_{\partial \Omega }|Df|^{p}d\sigma _{x}+\underset{\Vert \alpha
\Vert =[\lambda ]}{\sum }\int_{\partial \Omega }\int_{\partial \Omega }\frac{%
|D^{\alpha }f\left( x\right) -D^{\alpha }f\left( y\right) |^{p}}{%
|x-y|^{n+\{\lambda \}p-1}}d\sigma _{x}d_{\sigma y}<\infty \}
\end{equation*}%
and hence norm is given by$\ $%
\begin{equation}
\Vert f\Vert _{W^{\lambda ,p}\left( \partial \Omega \right) }=\left( 
\underset{\Vert \alpha \Vert \leq \lbrack \lambda ]}{\sum }\int_{\partial
\Omega }|Df|^{p}d\sigma _{x}+\underset{\Vert \alpha \Vert =[\lambda ]}{\sum }%
\int_{\partial \Omega }\int_{\partial \Omega }\frac{|D^{\alpha }f\left(
x\right) -D^{\alpha }f\left( y\right) |^{p}}{|x-y|^{n+\{\lambda \}p-1}}%
d\sigma _{x}d_{\sigma y}\right) ^{\frac{1}{p}}  \label{Slobodeckij 2 Norm}
\end{equation}
\end{description}

\ \ \ \ \ \ \ 

In the definitions of the Slobodeckij spaces and associated norms, the
irregularity exponent $n+\{\lambda \}p-1$ is due to the fact that the
dimension of $\partial \Omega $ is $n-1$ and $d\sigma $ is a hypersurface
measure on $\partial \Omega $.

\ 

Slobodeckij spaces as subspaces of Sobolev spaces but with fractional
exponents are analogues of the H\"{o}lder spaces in classical spaces of
continuous functions.

\ 

\section{Some Properties and Relations Between $D,\protect\zeta _{\Omega }$,$%
\protect\tau $ and $\protect\xi _{\partial \Omega }$}

\ 

\begin{proposition}
$D:W^{k,p}\left( \Omega ,Cl_{n}\right) \longrightarrow W^{k-1,p}\left(
\Omega ,Cl_{n}\right) $ is continuous with

\begin{equation*}
\Vert Df\Vert _{W^{k-1,p}\left( \Omega ,Cl_{n}\right) }\leq \gamma \Vert
f\Vert _{W^{k,p}\left( \Omega ,Cl_{n}\right) }
\end{equation*}%
for $\gamma =\gamma \left( n,p,\Omega \right) $ a positive constant.
\end{proposition}

\ 

\begin{proof}
Let $f\in W^{k,p}\left( \Omega ,Cl_{n}\right) $. We need to show that

\begin{equation*}
\Vert Df\Vert _{W^{k-1,p}\left( \Omega ,Cl_{n}\right) }\leq c\Vert f\Vert
_{W^{k,p}\left( \Omega ,Cl_{n}\right) }
\end{equation*}%
for some positive constant $c$.

\begin{equation*}
f\in W^{k,p}\left( \Omega ,Cl_{n}\right) \Longrightarrow \Vert f\Vert
_{W^{k,p}\left( \Omega ,Cl_{n}\right) }=\left( \underset{\Vert \alpha \Vert
\leq k}{\sum \int_{\Omega }}|D^{\alpha }f|^{p}dx\right) ^{\frac{1}{p}}<\infty
\end{equation*}%
But then

\begin{eqnarray*}
\Vert Df\Vert _{W^{k-1,p}\left( \Omega ,Cl_{n}\right) } &=&\left( \underset{%
\Vert \alpha \Vert \leq k-1}{\sum \int_{\Omega }}|D^{\alpha }f|^{p}dx\right)
^{\frac{1}{p}} \\
&\leq &\left( \underset{\Vert \alpha \Vert \leq k-1}{\sum \int_{\Omega }}%
|D^{\alpha }f|^{p}dx+\underset{\Vert \alpha \Vert =k-1}{\sum \int_{\Omega }}%
|D^{\alpha }f|^{p}dx\right) ^{\frac{1}{p}} \\
&=&\left( \underset{\Vert \alpha \Vert \leq k}{\sum \int_{\Omega }}%
|D^{\alpha }f|^{p}dx\right) ^{\frac{1}{p}} \\
&=&\Vert f\Vert _{W^{k,p}\left( \Omega ,Cl_{n}\right) }
\end{eqnarray*}

Therefore for $c=1$, the proposition is proved.
\end{proof}

\ \ \ \ \ \ \ \ \ \ \ \ \ \ \ \ \ \ \ 

\begin{proposition}
$D:\tciLaplace ^{p}\left( \Omega \right) \longrightarrow W^{-1,p}\left(
\Omega \right) $ is continuous for $1<p<\infty $.
\end{proposition}

\ 

\begin{proof}
Let $f\in \tciLaplace ^{p}\left( \Omega \right) $. Then

\begin{equation*}
\Vert Df\Vert _{W^{-1,p}\left( \Omega \right) }=\sup \{\frac{|\langle
Df,v\rangle |}{\Vert v\Vert _{W_{0}^{1,q}\left( \Omega \right) }}:v\neq
0,v\in W_{0}^{1,q}\left( \Omega \right) \}
\end{equation*}%
for $p^{-1}+q^{-1}=1$.

\ 

But 
\begin{equation*}
|\langle Df,v\rangle |=|\langle f,Dv\rangle |\leq \Vert f\Vert _{\tciLaplace
^{p}\left( \Omega \right) }\Vert Dv\Vert _{\tciLaplace ^{q}\left( \Omega
\right) }\leq \Vert f\Vert _{\tciLaplace ^{p}\left( \Omega \right) }\Vert
v\Vert _{W_{0}^{1,q}\left( \Omega \right) }
\end{equation*}

Thus by the Cauchy-Schwartz inequality we have

\begin{equation*}
\frac{|\langle Df,v\rangle |}{\Vert v\Vert _{W_{0}^{1,q}\left( \Omega
\right) }}\leq \frac{\Vert f\Vert _{\tciLaplace ^{p}\left( \Omega \right)
}\Vert v\Vert _{W_{0}^{1,q}\left( \Omega \right) }}{\Vert v\Vert
_{W_{0}^{1,q}\left( \Omega \right) }}=\Vert f\Vert _{\tciLaplace ^{p}\left(
\Omega \right) }
\end{equation*}%
Therefore

\begin{eqnarray*}
\Vert Df\Vert _{W^{-1,p}\left( \Omega \right) } &=&\sup \{\frac{|\langle
Df,v\rangle |}{\Vert v\Vert _{W_{0}^{1,q}\left( \Omega \right) }}:v\neq 0,%
\text{ }v\in W_{0}^{1,q}\left( \Omega \right) \} \\
&\leq &\sup \{\frac{\Vert f\Vert _{\tciLaplace ^{p}\left( \Omega \right)
}\Vert v\Vert _{W_{0}^{1,q}\left( \Omega \right) }}{\Vert v\Vert
_{W_{0}^{1,q}\left( \Omega \right) }}:v\neq 0,\text{ }v\in W_{0}^{1,q}\left(
\Omega \right) \} \\
&=&\Vert f\Vert _{\tciLaplace ^{p}\left( \Omega \right) }
\end{eqnarray*}

\ \ 
\end{proof}

\ \ \ \ \ \ \ \ \ \ \ \ \ \ \ \ 

\begin{proposition}
(Mapping properties) (\cite{gksp1}, \cite{dr1})

\ 

Let $k\in 
\mathbb{N}
\cup \{0\}$ and $1<p<\infty $. Then there are positive constants $\beta
=\beta \left( n,p,\Omega \right) $, $\theta =\theta \left( n,p,\Omega
\right) $ and $\delta =\delta \left( n,p,\Omega \right) $ such that

\ 
\begin{equation}
\zeta _{\Omega }:W^{k,p}\left( \Omega ,Cl_{n}\right) \longrightarrow
W^{k+1,p}\left( \Omega ,Cl_{n}\right)  \label{Theodorescu property}
\end{equation}%
with%
\begin{equation*}
\Vert \zeta _{\Omega }f\Vert _{W^{k+1,p}\left( \Omega ,Cl_{n}\right) }\leq
\beta \Vert f\Vert _{W^{k,p}\left( \Omega ,Cl_{n}\right) }
\end{equation*}

\begin{equation}
\xi _{\partial \Omega }:W^{\lambda ,p}\left( \partial \Omega ,Cl_{n}\right)
\longrightarrow W^{\lambda +\frac{1}{p},p}\left( \Omega ,Cl_{n}\right)
\label{Feuter property}
\end{equation}%
with 
\begin{equation*}
\Vert \xi _{\partial \Omega }f\Vert _{W^{\lambda +\frac{1}{p},p}\left(
\Omega ,Cl_{n}\right) }\leq \theta \Vert f\Vert _{W^{\lambda ,p}\left(
\partial \Omega ,Cl_{n}\right) }
\end{equation*}

and 
\begin{equation}
\tau :W^{k,p}(\Omega ,Cl_{n})\longrightarrow W^{k-\frac{1}{p},p}\left(
\partial \Omega ,Cl_{n}\right)  \label{trace property}
\end{equation}%
is the trace operator with

\begin{eqnarray*}
&&\underset{\Vert \alpha \Vert \leq \lbrack \lambda +\frac{1}{p}]}{\sum }%
\int_{\Omega }|D^{\alpha }\tau f|^{p}dx+\underset{\Vert \alpha \Vert
=[\lambda +\frac{1}{p}]}{\sum }\int_{\Omega }\int_{\Omega }\frac{|D^{\alpha
}\tau f(x)-D^{\alpha }\tau f(y)|^{p}}{|x-y|^{n+\{\lambda +\frac{1}{p}\}p}}%
dxdy \\
&\leq &\delta ^{p}\left( \underset{\Vert \alpha \Vert \leq \lbrack \lambda ]}%
{\sum }\int_{\partial \Omega }|D^{\alpha }f|^{p}dx+\underset{\Vert \alpha
\Vert =[\lambda ]}{\sum }\int_{\partial \Omega }\int_{\partial \Omega }\frac{%
|D^{\alpha }f(x)-D^{\alpha }f(y)|^{p}}{|x-y|^{n-1+\{\lambda +\frac{1}{p}\}p}}%
d\sigma _{x}d\sigma _{y}\right)
\end{eqnarray*}
\end{proposition}

\ \ \ \ \ \ \ \ \ \ \ \ \ \ \ \ 

\begin{proposition}
The composition $\xi _{\partial \Omega }\circ \tau $ preserves regularity of
a function in a Sobolev space.
\end{proposition}

\ \ \ \ \ \ \ \ \ \ \ \ \ \ \ \ \ \ \ \ 

\begin{proof}
Indeed, $\tau $ makes a function to loose a regularity fractional exponent
of $\frac{1}{p}$ when taken along the boundary of the domain. But the
boundary or \textit{Feuter} integral $\xi _{\partial \Omega }$ augments the
regularity exponent of a function defined on the boundary by an exponent of $%
\frac{1}{p}$.

\ \ \ \ \ \ \ \ \ \ \ \ \ \ \ \ \ \ \ \ \ \ \ \ \ \ \ \ \ \ \ \ \ \ \ \ \ \
\ \ \ \ \ 

Therefore the composition operator $\xi _{\partial \Omega }\circ \tau $
preserves or fixes the regularity exponent of a function in a Sobolev space.

\ \ \ 
\end{proof}

\begin{proposition}
(Borel-Pompeiu )

Let $\ f\in W^{k,p}\left( \Omega ,Cl_{n}\right) .$ Then

\begin{equation*}
f=\xi _{\partial \Omega }\tau f+\zeta _{\Omega }Df
\end{equation*}
\end{proposition}

\begin{corollary}
(i) If $f\in W_{0}^{k,p}\left( \Omega ,Cl_{n}\right) $, then 
\begin{equation*}
f=\zeta _{\Omega }Df
\end{equation*}%
That is $D$ is a right inverse for $\zeta _{\Omega }$ and $\zeta _{\Omega }$
is a left inverse for $D$ over traceless spaces.

(ii) If $f$ is monogenic function over $\Omega $, then%
\begin{equation*}
f=\xi _{\partial \Omega }\tau f
\end{equation*}

Therefore monogenic functions are always Cauchy transforms of their traces
over the boundary.
\end{corollary}

\ \ \ \ \ \ \ \ \ \ \ \ \ \ \ \ \ \ \ \ \ \ \ \ \ \ \ \ \ 

\section{Elliptic First and Second Order BVPs}

\ 

\begin{proposition}
Let $\ f\in W^{k-1,p}\left( \Omega ,Cl_{n}\right) $ for $k\geq 1$. Then the
first order elliptic BVP:

\begin{equation}
\left\{ 
\begin{array}{c}
Du=f\text{ \ in }\Omega \\ 
\tau u=g\text{ on }\partial \Omega%
\end{array}%
\right.  \label{BVP 1}
\end{equation}

has a solution $u\in W^{k,p}\left( \Omega ,Cl_{n}\right) $ given by 
\begin{equation*}
u\left( x\right) =\xi _{\partial \Omega }g+\zeta _{\Omega }f
\end{equation*}
\end{proposition}

\ \ \ \ \ \ \ \ \ 

\begin{proof}
The proof follows from the Borel-Pompeiu relation. As to where exactly $u$
and $g$ belong, we make the argument : $f$ \ is in $W^{k-1,p}\left( \Omega
,Cl_{n}\right) $ and hence from the mapping property of $D$, we have $u$ to
be a function in $W^{k,p}\left( \Omega ,Cl_{n}\right) $.

\ 

Also from the mapping property of the trace operator $\tau $ we have 
\begin{equation*}
\tau u=u|_{\partial \Omega }=g\in W^{k-\frac{1}{p},p}\left( \partial \Omega
,Cl_{n}\right)
\end{equation*}
\end{proof}

\ \ \ \ 

\begin{proposition}
The solution $u\in W^{k,p}\left( \Omega ,Cl_{n}\right) $ has a norm estimate
:

\begin{eqnarray*}
\Vert u\Vert _{W^{k,p}\left( \Omega ,Cl_{n}\right) } &\leq &\gamma
_{1}\left( \underset{\Vert \alpha \Vert \leq k-1}{\sum }\int_{\partial
\Omega }|D^{\alpha }g|^{p}d\sigma x+\underset{\Vert \alpha \Vert =k-1}{\sum }%
\int_{\partial \Omega }\int_{\partial \Omega }\frac{|D^{\alpha }g\left(
x\right) -D^{\alpha }g\left( y\right) |^{p}}{|x-y|^{n+p-2}}d\sigma
_{x}d\sigma _{y}\right) ^{\frac{1}{p}} \\
&&+\gamma _{2}\left( \underset{\Vert \alpha \Vert =k-1}{\sum }\int_{\partial
\Omega }|f|^{p}dx\right) ^{\frac{1}{p}}
\end{eqnarray*}

where $\gamma _{1},\gamma _{2}$ are constants the depend on $p$,$n$ and $%
\Omega $.
\end{proposition}

\ \ \ \ \ \ \ \ \ \ \ \ \ \ \ \ \ \ \ \ \ \ \ \ \ \ \ \ \ \ \ 

\begin{proof}
First let us determine regularity exponents of 
\begin{equation*}
g\in W^{k-\frac{1}{p},p}\left( \partial \Omega ,Cl_{n}\right) 
\end{equation*}%
For the regularity index $k-\frac{1}{p}$ the integer part is 
\begin{equation*}
\lbrack k-\frac{1}{p}]=k-1
\end{equation*}%
and $\ $the fractional part is 
\begin{equation*}
\{k-\frac{1}{p}\}=1-\frac{1}{p}
\end{equation*}%
Besides $\dim \left( \partial \Omega \right) =n-1$. From the mapping
properties of $D$, $\zeta _{\Omega }$, $\tau $ and $\xi _{\partial \Omega }$%
, we have%
\begin{equation*}
u\in W^{k,p}\left( \Omega ,Cl_{n}\right) 
\end{equation*}%
and 
\begin{equation*}
\tau u=g\in W^{k-\frac{1}{p},p}\left( \partial \Omega ,Cl_{n}\right) 
\end{equation*}%
Therefore the solution $u$ given by: 
\begin{equation*}
u\left( x\right) =\xi _{\partial \Omega }g+\zeta _{\Omega }f
\end{equation*}%
has norm estimate%
\begin{eqnarray*}
\Vert u\Vert _{W^{k,p}\left( \Omega ,Cl_{n}\right) } &=&\Vert \xi _{\partial
\Omega }g+\zeta _{\Omega }f\Vert _{W^{k,p}\left( \Omega ,Cl_{n}\right) } \\
&\leq &\Vert \xi _{\partial \Omega }g\Vert _{W^{k,p}\left( \Omega
,Cl_{n}\right) }+\Vert \zeta _{\Omega }f\Vert _{W^{k,p}\left( \Omega
,Cl_{n}\right) } \\
&\leq &\gamma _{1}\Vert g\Vert _{W^{k-\frac{1}{p},p}\left( \partial \Omega
,Cl_{n}\right) }+\gamma _{2}\Vert f\Vert _{W^{k-1,p}\left( \Omega
,Cl_{n}\right) } \\
&=&\gamma _{1}\left( \underset{\Vert \alpha \Vert \leq k-1}{\sum }%
\int_{\partial \Omega }|D^{\alpha }g|^{p}d\sigma x+\underset{\Vert \alpha
\Vert =k-1}{\sum }\int_{\partial \Omega }\int_{\partial \Omega }\frac{%
|D^{\alpha }g\left( x\right) -D^{\alpha }g\left( y\right) |^{p}}{%
|x-y|^{n-1+\{k-\frac{1}{p}\}p}}d\sigma _{x}d\sigma _{y}\right) ^{\frac{1}{p}}
\\
&&+\gamma _{2}\left( \underset{\Vert \alpha \Vert =k-1}{\sum }\int_{\partial
\Omega }|f|^{p}dx\right) ^{\frac{1}{p}} \\
&=&\gamma _{1}\left( \underset{\Vert \alpha \Vert \leq k-1}{\sum }%
\int_{\partial \Omega }|D^{\alpha }g|^{p}d\sigma x+\underset{\Vert \alpha
\Vert =k-1}{\sum }\int_{\partial \Omega }\int_{\partial \Omega }\frac{%
|D^{\alpha }g\left( x\right) -D^{\alpha }g\left( y\right) |^{p}}{%
|x-y|^{n+p-2}}d\sigma _{x}d\sigma _{y}\right) ^{\frac{1}{p}} \\
&&+\gamma _{2}\left( \underset{\Vert \alpha \Vert =k-1}{\sum }\int_{\partial
\Omega }|f|^{p}dx\right) ^{\frac{1}{p}}
\end{eqnarray*}%
The constants $\gamma _{1}$ and $\gamma _{2}$ are from the mapping
properties of $\xi _{\partial \Omega },\zeta _{\Omega }$ and $\tau $.
\end{proof}

\ \ \ \ \ \ \ \ \ \ \ \ \ \ 

\begin{proposition}
Let $f\in W^{k,p}\left( \Omega ,Cl_{n}\right) $. Then the second order
elliptic BVP 
\begin{equation}
\left\{ 
\begin{array}{c}
-\Delta u=f\text{ \ in }\Omega \\ 
\tau Du=g_{1}\text{ \ on }\partial \Omega \\ 
\tau u=g_{2}\text{ \ on }\partial \Omega%
\end{array}%
\right.  \label{bvp2}
\end{equation}%
has a solution given by 
\begin{equation*}
u=\xi _{\partial \Omega }\left( g_{2}\right) +\zeta _{\Omega }\xi _{\partial
\Omega }\left( g_{1}\right) +\zeta _{\Omega }\circ \zeta _{\Omega }\left(
f\right)
\end{equation*}%
in $W^{k+2,p}\left( \Omega \right) $ with 
\begin{equation*}
g_{1}\in W^{k+1-\frac{1}{p},p}\left( \partial \Omega \right) ,\text{ \ \ }%
g_{2}\in W^{k+2-\frac{1}{p},p}\left( \partial \Omega \right)
\end{equation*}
\end{proposition}

\ \ \ \ 

\begin{proof}
As $f\in W^{k,p}\left( \Omega ,Cl_{n}\right) $, the solution $u$ is in the
Sobolev space $W^{k+2,p}\left( \Omega \right) $ and hence 
\begin{equation*}
\tau u=g_{2}\in W^{k+2-\frac{1}{p},p}\left( \partial \Omega \right)
\end{equation*}%
But then $Du$ is in $W^{k+1,p}\left( \Omega \right) $ and hence 
\begin{equation*}
\tau Du=g_{1}
\end{equation*}%
is in the Slobodeckij space $W^{k+1-\frac{1}{p},p}\left( \partial \Omega
\right) $.

\ 

The solution $u$ of the BVP is obtained by repeated application of the
Borel-Pompeiu formula by writing the Laplacian $\Delta $ as $-D^{2}$.

\ 

Now let us first determine the integer and fractional parts of indices $k+2-%
\frac{1}{p}$ and \ $k+1-\frac{1}{p}$ \ as \ 
\begin{eqnarray*}
\lbrack k+2-\frac{1}{p}] &=&k+1,\text{ \ }\{k+2-\frac{1}{p}\}=1-\frac{1}{p}
\\
\lbrack k+1-\frac{1}{p}] &=&k,\text{ \ }\{k+1-\frac{1}{p}\}=1-\frac{1}{p}
\end{eqnarray*}

\ 

Therefore from the properties of the mappings studied above, we have a norm
estimate of the solution $u$ in $W^{k+2,p}\left( \Omega \right) $ in terms
of norms of $f$, $g_{1}$ and $g_{2}$ as follow:

\begin{eqnarray*}
\Vert u\Vert _{W^{k+2,p}\left( \Omega \right) } &=&\Vert \xi _{\partial
\Omega }\left( g_{2}\right) +\zeta _{\Omega }\xi _{\partial \Omega }\left(
g_{1}\right) +\zeta _{\Omega }\circ \zeta _{\Omega }\left( f\right) \Vert
_{W^{k+2,p}\left( \Omega \right) } \\
&\leq &\gamma _{1}\left( \underset{\Vert \alpha \Vert \leq k+1}{\sum }%
\int_{\partial \Omega }|D^{\alpha }g_{2}|^{p}d\sigma _{x}+\underset{\Vert
\alpha \Vert =k+1}{\sum }\int_{\partial \Omega }\int_{\partial \Omega }\frac{%
|D^{\alpha }g_{2}\left( x\right) -D^{\alpha }g_{2}\left( y\right) |^{p}}{%
|x-y|^{n+p-2}}d\sigma _{x}d\sigma _{y}\right) ^{\frac{1}{p}} \\
&&+\gamma _{2}\left( \underset{\Vert \alpha \Vert \leq k}{\sum }%
\int_{\partial \Omega }|D^{\alpha }g_{1}|^{p}d\sigma _{x}+\underset{\Vert
\alpha \Vert =k}{\sum }\int_{\partial \Omega }\int_{\partial \Omega }\frac{%
|D^{\alpha }g_{1}\left( x\right) -D^{\alpha }g_{1}\left( y\right) |^{p}}{%
|x-y|^{n+p-2}}d\sigma _{x}d\sigma _{y}\right) ^{\frac{1}{p}} \\
&&+\gamma _{3}\left( \underset{\Vert \alpha \Vert \leq k}{\sum }%
\int_{\partial \Omega }|D^{\alpha }f|^{p}dx\right) ^{\frac{1}{p}}
\end{eqnarray*}

for some positive constants $\gamma _{1},\gamma _{2}$ and $\gamma _{3}$ that
depend on $p,n,\Omega $
\end{proof}

\ \ \ \ \ \ \ \ 

\begin{proposition}
For the BVP $\left( \ref{BVP 1}\right) $ there exist positive constants $%
c,\gamma _{1}$ and $\gamma _{2}$ such that the solution $u\in
W^{k,2n}(\Omega )$ satisfies the norm estimate:

\begin{eqnarray*}
&&c^{-1}\left( \underset{\underset{x\neq y}{x,y\in \Omega }}{\sup }\frac{%
|u\left( x\right) -u\left( y\right) |}{|x-y|^{\frac{1}{2}}}+\Vert u\Vert
_{C\left( \Omega \right) }\right) \\
&\leq &\gamma _{1}\left( \underset{\Vert \alpha \Vert \leq k-1}{\sum }%
\int_{\partial \Omega }|D^{\alpha }g|^{2n}d\sigma x+\underset{\Vert \alpha
\Vert =k-1}{\sum }\int_{\partial \Omega }\int_{\partial \Omega }\frac{%
|D^{\alpha }g\left( x\right) -D^{\alpha }g\left( y\right) |^{2n}}{%
|x-y|^{n+p-2}}d\sigma _{x}d\sigma _{y}\right) ^{\frac{1}{2n}} \\
&&+\gamma _{2}\left( \underset{\Vert \alpha \Vert =k-1}{\sum }\int_{\partial
\Omega }|f|^{2n}dx\right) ^{\frac{1}{2n}}
\end{eqnarray*}
\end{proposition}

\begin{proof}
From the Sobolev embeding theorems, if $p>n,$ then 
\begin{equation*}
W^{k,p}\left( \Omega \right) \hookrightarrow C^{0,\lambda }\left( \Omega
\right)
\end{equation*}%
for $0<\lambda \leq 1-\frac{n}{p}$.

\ 

But then for $p=2n$, we have $0<\lambda \leq \frac{1}{2}$ and thuerefore the
solution $u$ which is in $W^{k,2n}\left( \Omega \right) $ is contained in H%
\"{o}lder spaces $C^{0,\lambda }\left( \Omega \right) $.

\ 

Thus $\exists $ $c=c(p$,$n$,$\Omega )>0$ such that 
\begin{equation*}
c^{-1}\Vert u\Vert _{C^{0,\lambda }\left( \Omega \right) }\leq \Vert u\Vert
_{W^{k,2n}\left( \Omega \right) }
\end{equation*}

That is 
\begin{eqnarray*}
&&c^{-1}\left( \underset{\underset{x\neq y}{x,y\in \Omega }}{\sup }\frac{%
|u\left( x\right) -u\left( y\right) |}{|x-y|^{\lambda }}+\Vert u\Vert
_{C\left( \Omega \right) }\right) \\
&\leq &\Vert u\Vert _{W^{k,2n}\left( \Omega \right) } \\
&\leq &\gamma _{1}\left( \underset{\Vert \alpha \Vert \leq k-1}{\sum }%
\int_{\partial \Omega }|D^{\alpha }g|^{2n}d\sigma x+\underset{\Vert \alpha
\Vert =k-1}{\sum }\int_{\partial \Omega }\int_{\partial \Omega }\frac{%
|D^{\alpha }g\left( x\right) -D^{\alpha }g\left( y\right) |^{2n}}{%
|x-y|^{n+p-2}}d\sigma _{x}d\sigma _{y}\right) ^{\frac{1}{2n}} \\
&&+\gamma _{2}\left( \underset{\Vert \alpha \Vert =k-1}{\sum }\int_{\partial
\Omega }|f|^{2n}dx\right) ^{\frac{1}{2n}}
\end{eqnarray*}

Choosing $\lambda =\frac{1}{2}$, we have the required result.
\end{proof}

\ \ \ \ \ \ \ \ \ \ \ \ \ \ \ \ \ \ \ \ \ \


\begin{thebibliography}{99}
\bibitem{sbern} S. Bernstein, \textit{Operator Calculus for Elliptic
Boundary Value Problems in Unbounded Domains}, Zeitschrift fur Analysis Und
ihre Anwendungen Vol.10 (1991) 4, 447-460.

\bibitem{bra1} F. Brackx, R. Delanghe and F. Sommen, \textit{Clifford
Analysis}, \textit{R}esearch Notes in Mathematics No.76, Pitman , London
1982.

\bibitem{dinezza} Di. Nezza et al., \textit{Hitchhiker's guide to the
fractional Sobolev spaces}, Bull. sci. math. (2012)

\bibitem{gksp1} K. G$\overset{..}{u}$rlebeck, U. K$\overset{..}{a}$hler, J.
Ryan and W. Spr$\overset{..}{o}$essig, \textit{Clifford Analysis Over
Unbounded Domains}, Adv. in Appl. Mathematics 19(1997), 216-239.

\bibitem{d1} Dejenie A. Lakew, $W^{2,k}-$\textit{Best Approximation of a }$%
\gamma -$\textit{Regular Function}, Journal of Applied Analysis, Vol. 13,
No. 2 (2007) pp. 259-273.

\bibitem{dr1} Dejenie A. Lakew and John Ryan, \textit{Clifford Analytic
Complete Function Systems for Unbounded Domains}, Math. Meth. in the Appl.
Sci. 2002;25;1527-1539 (with John Ryan).

\bibitem{dr2} \_\_\_\_, \textit{Complete Function Systems and Decomposition
Results Arising in Clifford Analysis}, Computational Methods and Function
Theory,\ CMFT No. 1(2002) 215-228\textit{\ }(with John Ryan).

\bibitem{el} Evans, Lawrence, \textit{Partial Differential Equations},
Graduate Studies in Mathematics, Vol. 19, AMS, 1998

\bibitem{mp} S.G. Mikhlin, S. Prossdorf, \textit{Singular Integral Operators}%
, Academic Verlag, Berlin (1980)\textit{.}

\bibitem{r1} John Ryan, \textit{Intrinsic Dirac Operators in }$C^{n}$,
Advances in Mathematics 118, 99-133(1996).

\bibitem{sm} K.T. Smith, \textit{Primier of Modern Analysis}, Undergraduate
Texts in Mathematics, Springer Verlag, New York (1983).

\bibitem{tre} H. Triebel, \textit{Interpolation Theory, Function Spaces,
Differential Operators}, North-Holland Mathematical Library, 1978.
\end{thebibliography}
\end{document}